\documentclass[12pt]{article}
\usepackage{amsmath}
\usepackage[usenames]{color}
\usepackage{mathrsfs}
\usepackage{amsfonts}
\usepackage{amssymb,amsmath}
\usepackage{CJK}
\usepackage{cite}
\usepackage{cases}
\usepackage{amsthm}

\def\cl{\centerline}

\def\al{\alpha}
\def\vs{\vspace*}

\def\Z{\mathbb{Z}}

\def\F{\mathbb{F}}
\def\QED{\hfill$\Box$}
\def\W{\mathscr W(\Gamma,s)}

\numberwithin{equation}{section}
\newtheorem{theo}{Theorem}[section]

\newtheorem{lemm}[theo]{Lemma}

\newtheorem{clai}{Claim}

\begin{document}
\begin{CJK*}{GBK}{song}

\begin{center}
{\bf\large 
Derivations, automorphisms and second cohomology of generalized loop  Schr\"{o}dinger-Virasoro
algebras\,$^*$} \footnote {$^*\,$Supported by NSF grant no. 11371278, 11431010, 
the Fundamental Research Funds for the Central Universities of China, Innovation Program of Shanghai Municipal Education Commission and  Program for Young Excellent Talents in Tongji University.

$\ \ ^{{\,}\dag}\,$Corresponding author: J. Han (jzhan@tongji.edu.cn). }
\end{center}

\cl{Haibo Chen,  Guangzhe Fan, Jianzhi
Han$^{\dag}$, Yucai Su
}

\cl{\small  Department of Mathematics, Tongji University,
Shanghai 200092, China}
\vs{5pt}

\vs{8pt}

{\small
\parskip .005 truein
\baselineskip 3pt \lineskip 3pt

\noindent{{\bf Abstract:} The derivation
algebras,  automorphism groups and  second cohomology groups
 of the generalized loop Schr\"{o}dinger-Virasoro algebras are completely determined in this paper. \vs{5pt}

\noindent{\bf Key words:} generalized loop Schr\"{o}dinger-Virasoro
 algebra, derivation, automorphism, 2-cocycle.}

\noindent{\it Mathematics Subject Classification (2010):} 17B05, 17B40, 17B65, 17B68.}
\parskip .001 truein\baselineskip 6pt \lineskip 6pt
\section{Introduction}
\setcounter{equation}{0}

It is well-known that the Schr\"{o}dinger-Virasoro algebra is an
important infinite dimensional Lie algebra, whose theory plays a
crucial role in many areas of mathematics and physics. The original
Schr\"{o}dinger-Virasoro Lie algebra was introduced in \cite{M1}, in
the context of nonequilibrium statistical physics, containing as
subalgebras both the Lie algebra of invariance of the free
Schr\"{o}dinger equation and the Virasoro algebra, whose Lie bialgebra structures and representations were investigated  (see, e.g., \cite{HLS,LS1,ZT}).  Let $\mathbb{F}$ be a field of characteristic 0, $\Gamma$
a proper additive  subgroup of $\mathbb{F}$,  and let $s\in \mathbb{F}$ be such
that $s\notin \Gamma$ and $2s\in \Gamma$. We denote
$\Gamma_1=s+\Gamma$ and $T=\Gamma\cup \Gamma_1.$ It is obvious that
$T$ is an additive subgroup of $\mathbb{F}$. The generalized
Schr\"{o}dinger-Virasoro algebra $\mathfrak{gsv}[\Gamma,s]$ is an infinite dimensional Lie
algebra with $\mathbb{F}$-basis $\{L_\alpha,M_\alpha,Y_{\alpha+s}
\mid \alpha\in \Gamma\}$ subject to the following relations:
\begin{eqnarray*}&[L_\alpha,L_\beta]= (\beta-\alpha)L_{\alpha+\beta},\ [L_\alpha,M_\beta]= \beta M_{\alpha+\beta},\
[L_\alpha,Y_{\beta+s}]= (\beta+s -\frac{\alpha}{2})
Y_{\alpha+\beta+s},\\
&[M_\alpha,M_\beta]=0,\ [M_\alpha,Y_{\beta+s}]=0,\ [Y_{\alpha+s},Y_{\beta+s}]= (\beta-\alpha)M_{\alpha+\beta+2s}.\end{eqnarray*}
Derivations, central extensions, automorphisms and Verma modules were studied in \cite{WLX,TZ},

The {\em generalized loop
Schr\"{o}dinger-Virasoro algebra} $\W$ is the tensor product \linebreak \mbox{$\mathfrak{gsv}[\Gamma,s]\otimes\mathbb{F}[t,t^{-1}]$} of the generalized
Schr\"{o}dinger-Virasoro algebra $\mathfrak{gsv}[\Gamma,s]$
with the  Laurent polynomial algebra $\mathbb{F}[t,t^{-1}]$, subject to relations:
\begin{eqnarray}
\label{rel1}&&[L_{\alpha,i},L_{\beta,j}]= (\beta-\alpha)L_{\alpha+\beta,i+j},
\ \ \ \ \ 
[L_{\alpha,i},M_{\beta,j}]= \beta M_{\alpha+\beta,i+j},\\
&&\label{rel3}[L_{\alpha,i},Y_{\beta+s,j}]= (\beta+s -\frac{\alpha}{2})
Y_{\alpha+\beta+s,i+j},\\
&&\label{rel4}[M_{\alpha,i},M_{\beta,j}]=0,
\ \ \ \ \ \ \ \ 
\label{rel5}[M_{\alpha,i},Y_{\beta+s,j}]=0,\\ &&\label{rel6}[Y_{\alpha+s,i},Y_{\beta+s,j}]=
(\beta-\alpha)M_{\alpha+\beta+2s,i+j}
\mbox{ \ for any $\alpha,\beta\in \Gamma$ and $i,j\in\Z$,}\end{eqnarray}
where in general, $X_{\gamma,i}=X_{\gamma}\otimes t^i$ (and denoted by $X_\gamma t^i$ for short) for any $X\in\{L, M, Y\},\gamma\in T$ and $i\in\Z$.  Here and below,
  we use the convention that if an undefined notation appears in an expression, we treat it as zero; for instance,
$L_{\alpha+s}=0,\,M_{\beta+s}=0$ and $Y_\alpha=0$  for any $\alpha\in\Gamma $.
Note that
$\oplus_{i\in\Z}
\mathbb F M_{0,i}$ is exactly the center of $\W$.

The structure theory  of  infinite dimensional Lie algebras
have been extensively studied  due to its important role in Lie algebra (see, e.g.,
\cite{CHS, DZ,F,FZY,SS041,SS042,S02-1,
S1,S2,SXZ00,SZ1,SJ,WLX1}). Loop algebras are certain types of Lie algebra, of particular
interest in theoretical physics, which contribute a large ingredient to
the structure theory of Lie algebras such as affine Lie
algebras in \cite{K}. Recently, generalized loop super-Virasoro  algebra, loop Witt algebra and generalized
loop Virasoro  algebra  were studied in \cite{DHS,TZ1,WWY,GLZ}.

In this paper, we  study  derivations,
automorphisms and  second cohomology groups of generalized loop Schr\"{o}dinger-Virasoro
algebras.
As a result, we also determine the universal central extension of $\W$ (cf.~\eqref{Ext-}).
To determine the derivation algebra ${\rm Der\,}\W$ of $\W$, the most effective way is to decompose it into a sum of
 the inner derivation algebra ad$\,\W$ and the derivation  subalgebra $\big({\rm Der\,}\W\big)_0$ consisting of all homogenous derivations of degree zero (see \cite{F} and Lemma \ref{lemm2.4}).   Then the remaining task is to determine $\big({\rm Der\,}\W\big)_0$, which is the main goal of Section 2. While determining  the second cohomology group of $\W$ depends heavily  on computation, the determination of the automorphism group ${\rm Aut\,}\W$ of $\W$ is the most difficult part among these three kinds of structures.  For the problem of determining ${\rm Aut\,}\W$,  it is easy but important to observe that the problem can be  reduced to the case  $\mathfrak{gsv}[\Gamma,s]$ when setting $t=1$. By this and by constructing some kinds of concrete automorphisms we  give an explicit characterization of the automorphism group of $\W$.
The main results of the present paper are summarized in Theorems \ref{theo1},  \ref{theo3.4} and \ref{th4.6}.

\section{Derivation algebra of $\W$}

   A linear map $D:\W\rightarrow \W$ is called a {\it
derivation} of $\W$ if

$$D\big([x,y]\big)=[D(x),y]+[x,D(y)],  \quad \forall\ x,y\in \W;$$ if in addition, $D={\rm ad}_z$ for some $z\in \W$,
then $D$ is called an {\em inner derivation}, where
 ${\rm ad}_z(x)=[z,x]$ for any $x\in \W$. Denote by ${\rm Der\,}\W$ and ${\rm ad\,}\W$ the vector space of all derivations and inner
derivations, respectively. Then
$${H}^{1}(\W, \W)\cong{\rm Der\,}{\W}/{\rm ad\,}{\W}$$ is {\em the first cohomology group of $\W$}.

Note  from the relations \eqref{rel1} and \eqref{rel3} that $L_{0,0}$ is semisimple, which gives
a $T$-grading on $\W$: $\W=\oplus_{\mu\in T}
\W_\mu,$
with\begin{eqnarray}\label{T-grading}
\W_\mu=\{x\in\W\mid [L_{0,0},x]=\mu x\}=\{L_{\mu,j},\,M_{\mu,j},\,Y_{\mu,j}\,|\,j\in\Z\}.
\end{eqnarray}We
say that a derivation $D\in {\rm Der\,}\W$ is of degree $\gamma\in T$
if   $D\big(\W_{\al}\big)\subset \W_{\al+\gamma}$ for all $\alpha\in T$.
Let $\big({\rm Der\,}\W\big)_{\gamma}$ be the space of all derivations
of degree $\gamma$.

The following result first appeared in \cite{S2}.
\begin{lemm}\label{lemm2.1}
For any derivation $D\in \W$, one has
 \begin{equation}\label{L2.1}
D=\mbox{$\sum\limits_{\gamma\in T}$}D_\gamma,\quad where \quad
D_{\gamma}\in\big({\rm Der\,}\W\big)_\gamma ,
\end{equation}
 which holds in the sense that for every $x\in\W$, only finitely
 many $D_\gamma(x)\neq0$,
 and $D(x)=\mbox{$\sum
 _{\gamma\in T}$}D_\gamma(x)$ $($such a sum in \eqref{L2.1} is called summable$)$.
\end{lemm}
\noindent{\it Proof.~}~For any $x_\alpha\in \W_\alpha$ and
$D\in{\rm Der\,}\W$, assume that
$D(x_\alpha)=\mbox{$\sum
_{\beta\in T}$} y_\beta.$ We
define $D_\gamma(x_\alpha)=y_{\alpha+\gamma}$. Then $D_\gamma$ is a
derivation by a direct computation.\QED

\begin{lemm}\label{lemm2.2}
If $\gamma\neq 0$, then every $D \in\big({\rm Der\,}{\W}\big)_{\gamma}$ is
an inner derivation.
\end{lemm}
\begin{proof}It follows from \eqref{T-grading} and applying $D$ to $[L_{0,0},x]=\mu x$ for any $x\in\W_\mu$ with $\mu\in T$ that
\begin{equation*}
\mu
D(x)=[D(L_{0,0}),x]+[L_{0,0},D(x)].
\end{equation*}
  From this, using  $D(x)\in \W_{\mu+\gamma}$ and  \eqref{T-grading}, one has
$D(x)={\rm ad}_{-\gamma^{-1}D(L_{0,0})}(x).$
Hence, $D={\rm ad}_{-\gamma^{-1}D(L_{0,0})},$ an inner derivation.\end{proof}

As a  consequence of  Lemmas
\ref{lemm2.1} and \ref{lemm2.2}, we immediately have the following result.

\begin{lemm}\label{lemm2.4}
$ {\rm Der\,} \W=\big({\rm Der\,}\W\big)_0+{\rm ad\,} \W.$
\end{lemm}
\begin{proof}
For $\gamma\ne0$, assume $D_\gamma=
{\rm ad}_{u_\gamma}$ for some 
$u_\gamma
\in \W_\gamma$. 
If $\{\gamma\ne0\,|\,u_\gamma\ne0\}$ is an infinite set, then by \eqref{T-grading},
$D(L_{0,0})=D_0(L_{0,0})-\sum_{\gamma\ne0}\gamma u_\gamma$ is an infinite sum which is not an element of $\W$, a contradiction.
\end{proof}
Next we are going to characterize $({\rm Der\,}\W)_0$. To do this, let us first introduce some notations. For any map $f$ from   a set $A$ to another set $B$, we write the image of $a\in A$ under $f$ as $f_a$ rather than $f(a)$ in what follows.
Denote ${\rm Hom}_\mathbb{Z}(T,\ \mathbb{F}[t,t^{-1}])$  the set of group
homomorphisms  $\phi:T\rightarrow \mathbb{F}[t,t^{-1}]$, which carries the structure of a vector space by defining $(c\phi)(\gamma)=c\phi(\gamma)$ for any given $\phi\in{\rm Hom}_\Z(T,\ \mathbb F[t,t^{-1}])$ and any $c\in\mathbb F$, $\gamma\in T$. Denote by $\mathbb{F}[t,t^{-1}]\frac{d}{dt}$ the  derivation algebra of
$\mathbb{F}[t,t^{-1}]$   and  $$\mathfrak{g}(\Gamma)=
 \big\{g: \Gamma\rightarrow  \mathbb F[t,t^{-1}]\mid (\beta-\alpha)g_{\alpha+\beta}=\beta
g_{\beta}-\alpha g_{\alpha}, \quad\forall\  \alpha,\beta\in\Gamma\big\}.$$

For any $\phi\in {\rm Hom}_\Z(T,\ \mathbb{F}[t,t^{-1}])$, $g\in \mathfrak{g}(\Gamma)$, $b\in\mathbb F[t,t^{-1}]$ and $\rho\in \mathbb{F}[t,t^{-1}]\frac{d}{dt}$,  define respectively four linear
 maps $D_\phi,D_g,D_b$ and $D^\rho$ of $\W$
 as follows:
\begin{eqnarray}
\label{red2.3}&& D_{\phi}( X_{\mu,i})=\phi(\mu)X_{\mu,i},\\
&&\label{red2.4}D_{g}(L_{\alpha,i})=g_{\alpha}M_{\alpha,i}, \ \ \ \
D_{g}(M_{\alpha,i})=D_{g}(Y_{\alpha+s,i})=0,\\
&&\label{red2.5}D_b(L_{\alpha,i})=0,\ \ \ \
D_b(M_{\alpha,i})=bM_{\alpha,i}, \ \ \ \
D_b(Y_{\alpha+s,i})=\frac{1}{2}bY_{\alpha+s,i},\\
&&\label{red2.6}D^{\rho}(
X_{\mu,i})=X_{\mu}\rho(t^{i}),\quad \forall\  X\in\{L,M,Y\}, \mu\in
T, \alpha\in\Gamma, i\in\Z.\end{eqnarray}  It is easy to see that these
four  operators defined above  are all homogeneous derivations
of degree zero.   We use the same notations    \begin{eqnarray*}&&\mathcal{D}_{{\rm
Hom}_\Z(T,\ \mathbb{F}[t,t^{-1}])}:=\{D_\phi\mid \phi\in {\rm
Hom}_\Z(T,\ \mathbb{F}[t,t^{-1}]) \},\\
&&\mathcal{D}_{\mathfrak{g}(\Gamma)}:={\rm span}\{D_g\mid g\in \mathfrak{g}(\Gamma) \},\\
&&\mathcal{D}_{\mathbb
F[t,t^{-1}]}:= \{D_b\mid b\in\mathbb F[t,t^{-1}]\}\quad{\rm and}\quad\mathcal{D}_{\mathbb{F}[t,t^{-1}]\frac{d}{dt}}:=\{D^\rho\mid \rho\in \mathbb{F}[t,t^{-1}]\frac{d}{dt}\}\end{eqnarray*} to
denote respectively the corresponding subspaces of ${\rm Der\,}\W.$

Take any $D\in ({\rm Der\,} \W)_0$. For any $\alpha\in\Gamma$ and
$i\in\Z$, assume that
\begin{equation}\label{dd7}
D(L_{\alpha,i})=f_{\alpha,i}L_{\alpha,i}+g_{\alpha,i}M_{\alpha,i},\quad {\rm where}\ f_{\alpha,i},g_{\alpha,i}\in
\mathbb{F}[t,t^{-1}].
\end{equation}
 Applying $D$ to the first relation in \eqref{rel1} and using  \eqref{dd7} give
\begin{eqnarray*}
&&(\beta-\alpha)f_{\alpha,i}L_{\alpha+\beta,i+j}-\alpha g_{\alpha,i}M_{\alpha+\beta,i+j}+(\beta-\alpha)f_{\beta,j}L_{\alpha+\beta,i+j}
+\beta g_{\beta,j}M_{\alpha+\beta,i+j}\\
&&=(\beta-\alpha)f_{\alpha+\beta,i+j}L_{\alpha+\beta,i+j}+(\beta-\alpha)g_{\alpha+\beta,i+j}M_{\alpha+\beta,i+j}.\end{eqnarray*}
 Comparing the coefficients of $L_{\alpha+\beta,i+j}$ and $M_{\alpha+\beta,i+j},$ one has
\begin{eqnarray}\label{eq1}
&&f_{\alpha+\beta,i+j}=f_{\alpha,i}+f_{\beta,j}
\mbox{ \ for } \alpha\neq \beta,
\\ 
\label{eq2}
&&(\beta-\alpha)g_{\alpha+\beta,i+j}=\beta
g_{\beta,j}-\alpha g_{\alpha,i},\quad \forall\ \al,\beta\in\Gamma, i,j\in\Z.
\end{eqnarray}
It is obvious that $f_{0,0}=0$. Furthermore, from
\eqref{eq1}  one can, in fact, see that
\begin{equation}\label{f=f+f}f_{\alpha+\beta,i+j}=f_{\alpha,i}+f_{\beta,j},\quad \forall\ \alpha,\beta\in \Gamma.\end{equation}

Setting $\alpha=0$ and $j=0$ in \eqref{eq2}, one has
$g_{\beta,i}=g_{\beta,0}$ whenever $\beta\neq 0$.  From this and by
setting $\alpha=-\beta\neq0$ in \eqref{eq2}, we have $g_{0,i+j}=g_{0,0}.$
Thus,  $g_{\beta,i}=g_{\beta,0}$ for any $ \beta\in \Gamma$ and $i \in
\Z, $ which is denoted by $g_\beta$ for convenience. Whence
\eqref{eq2} turns out to be $
(\beta-\alpha)g_{\alpha+\beta}=\beta g_{\beta}-\alpha g_{\alpha},
$ 
that is, $g\in \mathfrak{g}(\Gamma)$.

Put $\rho=tf_{0,1}\frac{d}{dt}$. Replacing  $D$ by $D-D^\rho$
entails us to assume that $f_{0,1}=0$, which in turn, by \cite[Lemma
2.4]{WWY}, gives rise to $f_{\beta,i}=f_{\beta,0}$ for all
$\beta\in\Gamma$ and $i\in\Z$. So in what follows we can simply write
$f_{\beta,i}$ as $f_\beta$. In this case, \eqref{f=f+f} becomes
\begin{equation}\label{ff+f}f_{\alpha+\beta}=f_{\alpha}+f_{\beta},\quad \forall\ \alpha,\beta\in \Gamma.
\end{equation}
In particular, $f\in{\rm Hom}_\Z(\Gamma, \mathbb F[t,t^{-1}])$. Whence we arrive at the following lemma.

\begin{lemm}\label{lemmas2.4-1} Let $D$ be as above.   Then there exist
$f\in {\rm Hom}_\mathbb{Z}(\Gamma,\ \mathbb{F}[t,t^{-1}])$ and  $g\in \mathfrak{g}(\Gamma)$ such that
$$D(L_{\alpha,i})=f_{\alpha}L_{\alpha,i}+g_{\alpha}M_{\alpha,i}
\mbox{ \ for
$\alpha\in\Gamma$ and $i\in\Z$.}$$\end{lemm}

While for  the expressions of  $D(M_{\alpha,i})$ and $D(Y_{\alpha+s,i}),$ we have the following result.

\begin{lemm}\label{lemmas2.5}
Let $D$ and $f$ as in Lemma $\ref{lemmas2.4-1}$. Then there exists
$b\in\mathbb F[t,t^{-1}]$ such that \begin{eqnarray*}&&
D(M_{\alpha,i})=(f_{\alpha}+b)M_{\alpha,i},
\  \ 
D(Y_{\alpha+s,i})=\frac{1}{2}
(f_{2(\alpha+s)}+b)Y_{\alpha+s,i}
\mbox{ \ for
$\alpha\in\Gamma$ and $i\in\Z$.}\end{eqnarray*}
\end{lemm}
\begin{proof}
For any $\alpha\in\Gamma$ and  $i\in\Z$, assume
\begin{equation}\label{eeq2.8}
D(M_{\alpha,i})=\psi_{\alpha,i}L_{\alpha,i}+\varphi_{\alpha,i}M_{\alpha,i},
\quad {\rm where}\
\psi_{\alpha,i},\varphi_{\alpha,i}\in
\mathbb{F}[t,t^{-1}].
\end{equation}Applying $D$ to
$[L_{-\alpha,-i},M_{\alpha,i}]=\alpha M_{0,0}$ and using \eqref{eeq2.8},  we
obtain
 $$\alpha(f_{-\alpha}+\varphi_{\alpha,i}-\varphi_{0,0})M_{0,0}+\alpha(2\psi_{\alpha,i}-\psi_{0,0})L_{0,0}=0.$$
Hence,
\begin{equation}\label{eq4}
f_{-\alpha}+\varphi_{\alpha,i}=\varphi_{0,0} \quad {\rm and} \quad
\psi_{\alpha,i}=\frac{1}{2}\psi_{0,0},\quad \forall\ 0\neq \alpha\in \Gamma.
\end{equation}
Since $M_{0,0}$ lies in the center of $\W$,  so does  $D(M_{0,0})$. This forces $\psi_{0,0}=0$ and
 thereby $\psi_{\alpha,i}=0$ for any $\alpha\in\Gamma$ and $i\in\Z$.    It follows from this fact,  \eqref{ff+f} and \eqref{eq4}  that \eqref{eeq2.8} can be simplified as
$D(M_{\alpha,i})=(f_{\alpha}+b)M_{\alpha,i}, {\rm\ where\ } b=\varphi_{0,0}.$ This is the first desired relation.

For arbitrary $\alpha\in\Gamma$ and $i\in\Z$, we assume that
\begin{equation}\label{equ2.10}
D(Y_{\alpha+s,i})=h_{\alpha+s,i}Y_{\alpha+s,i}\quad{\rm for\ some\ }
h_{\alpha+s,i}\in \mathbb{F}[t,t^{-1}].
\end{equation}
 Applying $D$ to \eqref{rel6}
 and by \eqref{equ2.10},
one has
\begin{eqnarray*}
(\beta-\alpha)(h_{\alpha+s,i}+h_{\beta+s,j})M_{\alpha+\beta+2s,i+j}
=(\beta-\alpha)(f_{\alpha+\beta+2s}+b)M_{\alpha+\beta+2s,i+j}.
\end{eqnarray*}
Hence,  $h_{\alpha+s,i}+h_{\beta+s,j}=
f_{\alpha+\beta+2s}+b$ for any  $\alpha\neq\beta\in\Gamma$. In fact, the condition $\alpha\neq \beta$ can be removed, namely,
\begin{equation}\label{eq2.6}
h_{\alpha+s,i}+h_{\beta+s,j}=
f_{\alpha+\beta+2s}+b,\quad \forall\ \alpha, \beta\in\Gamma.
\end{equation}From this we see that $h_{\alpha+s,j}$
is independent of the second subindex $j$, which is written as $h_{\alpha+s}$ for short. Moreover, setting $\beta=\alpha$ in \eqref{eq2.6}
we have
$
h_{\alpha+s}=\frac{1}{2}(f_{2(\alpha+s)}+b),
$ 
this is the second desired relation. The proof is complete.
\end{proof}
For any $f\in{\rm Hom}_\Z(\Gamma,\mathbb F[t,t^{-1}]) $ (the set of group homomorphisms from $\Gamma$ to $\mathbb F[t,t^{-1}])$, we extend it to a map from $T$ to $\mathbb F[t,t^{-1}]$ by
setting $
f_{\gamma}=
\frac{1}{2}f_{2\gamma}
\mbox{ if  }\gamma \in \Gamma_1.
$ 
Then  it is easy to verify the resulting new map  $f$ is a group homomorphism, i.e., $f\in {\rm Hom}_\Z(T,\ \mathbb F[t,t^{-1}])$.

\begin{theo}\label{theo1}Set $$\mathcal{D}^\prime_{{\rm
Hom}_\Z(T,\ \mathbb{F}[t,t^{-1}])}=\{D_\phi\mid \exists f(t)\in\mathbb F[t,t^{-1}]\ {\rm such\ that}\ \phi(\gamma)=\gamma f(t), \forall \gamma\in T\}.$$
Then the derivation space of $\W$ has the following decomposition
 $${\rm
Der\,}{\W}={\rm ad\,}\W+\Big(\mathcal{D}_{{\rm
Hom}_\Z(T,\ \mathbb{F}[t,t^{-1}])}\oplus  \mathcal{D}_{\mathfrak{g}(\Gamma)}\oplus \mathcal{D}_{\mathbb
F[t,t^{-1}]}\oplus \mathcal{D}_{\mathbb{F}[t,t^{-1}]\frac{d}{dt}}\Big)$$ and the first cohomology group ${H}^{1}(\W, \W)$ is linearly isomorphic to $$\big(\mathcal{D}_{{\rm
Hom}_\Z(T,\ \mathbb{F}[t,t^{-1}])}/\mathcal{D}^\prime_{{\rm
Hom}_\Z(T,\ \mathbb{F}[t,t^{-1}])}\big)\oplus  \mathcal{D}_{\mathfrak{g}(\Gamma)}\oplus \mathcal{D}_{\mathbb
F[t,t^{-1}]}\oplus \mathcal{D}_{\mathbb{F}[t,t^{-1}]\frac{d}{dt}}.$$
\end{theo}

\begin{proof}Let us first characterize the derivation algebra $\big({\rm Der\,}\W\big)_0$.
 \begin{clai}$\big({\rm Der\,}\W\big)_0=\mathcal{D}_{{\rm
Hom}_\Z(T,\ \mathbb{F}[t,t^{-1}])}+ \mathcal{D}_{\mathfrak{g}(\Gamma)}+ \mathcal{D}_{\mathbb
F[t,t^{-1}]}+\mathcal{D}_{\mathbb{F}[t,t^{-1}]\frac{d}{dt}}.$ \end{clai} To prove this, take any $D\in \big({\rm Der\,}\W\big)_0$. It follows from Lemmas \ref{lemmas2.4-1}, \ref{lemmas2.5} and the remarks before Lemma \ref{lemmas2.4-1} and this theorem that there exist some $\rho\in \mathbb{F}[t,t^{-1}]\frac{d}{dt},$
\mbox{$f\in{\rm Hom}_\Z(T,\ \mathbb F[t,t^{-1}]),\, g\in \mathfrak{g}(\Gamma)$} and $b\in \mathbb{F}[t,t^{-1}]$ such that the following relations hold:

\begin{eqnarray*}
&&(D-D^\rho)(L_{\alpha,i})=f_{\alpha}L_{\alpha,i}+g_{\alpha}M_{\alpha,i},
\ \ \ \
(D-D^\rho)(M_{\alpha,i})=(f_{\alpha}+b)M_{\alpha,i},\\
&&(D-D^\rho)(Y_{\alpha+s,i})=\frac{1}{2}
(f_{2(\alpha+s)}+b)Y_{\alpha+s,i},\quad \forall\ \alpha\in\Gamma,i\in\Z.
\end{eqnarray*}On the other hand, it follows from  the definition of $D_f, D_g$ and $D_b$  that
\begin{eqnarray*}
&&(D_f+D_g+D_b)(L_{\alpha,i})=f_{\alpha}L_{\alpha,i}+g_{\alpha}M_{\alpha,i},
\ \ \ \
(D_f+D_g+D_b)(M_{\alpha,i})=(f_{\alpha}+b)M_{\alpha,i},\\
&&(D_f+D_g+D_b)(Y_{\alpha+s,i})=\frac{1}{2}
(f_{2(\alpha+s)}+b)Y_{\alpha+s,i},\quad \forall\ \alpha\in\Gamma,i\in\Z.
\end{eqnarray*}
Thus, $D=D^\rho+D_f+D_g+D_b$.

\begin{clai}\label{claim2}
The sum ${D}_{{\rm
Hom}_\Z(T,\ \mathbb{F}[t,t^{-1}])}+ \mathcal{D}_{\mathfrak{g}(\Gamma)}+ \mathcal{D}_{\mathbb
F[t,t^{-1}]}+\mathcal{D}_{\mathbb{F}[t,t^{-1}]\frac{d}{dt}}$ is direct.
\end{clai} To prove the claim, let $\mathcal D=D_\phi+\mbox{$\sum
_{j\in J}$} c_jD_{g_j}+D_b+D^\rho$, where $J$ is a finite index set,  $\phi\in {{\rm
Hom}_\Z(T,\ \mathbb{F}[t,t^{-1}])},$ $g_j\in{\mathfrak{g}(\Gamma)}, b\in{\mathbb
F[t,t^{-1}]},  \rho\in{\mathbb{F}[t,t^{-1}]\frac{d}{dt}}$ and $c_j\in\mathbb F.$
Then the claim is equivalent to showing that $\mathcal D=0$   implies $$D_{\phi}=\mbox{$\sum\limits_{j\in J}$}c_jD_{g_j}=D_b=D^\rho=0.$$ But the latter follows immediately 
 by applying $ D$
to the basis elements $L_{\alpha,i}, M_{\al,i}$ and $Y_{\al+s,i}$ of $\W$ for arbitrary $\alpha\in\Gamma$ and $i\in\Z.$ So Claim \ref{claim2} holds.

Now the first statement  follows the Claims 1, 2 and Lemma \ref{lemm2.4}; while for the second statement it suffices  to show the following claim.
\begin{clai}\label{c-l-a-i3}The intersection
${\rm ad\,}\W\cap\Big(\mathcal{D}_{{\rm
Hom}_\Z(T,\ \mathbb{F}[t,t^{-1}])}\oplus  \mathcal{D}_{\mathfrak{g}(\Gamma)}\oplus \mathcal{D}_{\mathbb
F[t,t^{-1}]}\oplus \mathcal{D}_{\mathbb{F}[t,t^{-1}]\frac{d}{dt}}\Big)$ is equal to ${\rm ad\,}\W\cap\mathcal{D}_{{\rm
Hom}_\Z(T,\ \mathbb{F}[t,t^{-1}])}=\mathcal{D}^\prime_{{\rm
Hom}_\Z(T,\ \mathbb{F}[t,t^{-1}])}.$
\end{clai}
To prove the claim, let $
D\in \mathcal{D}_{{\rm
Hom}_\Z(T,\ \mathbb{F}[t,t^{-1}])}\oplus  \mathcal{D}_{\mathfrak{g}(\Gamma)}\oplus \mathcal{D}_{\mathbb
F[t,t^{-1}]}\oplus \mathcal{D}_{\mathbb{F}[t,t^{-1}]\frac{d}{dt}}$ have the form $D_{\phi}+\mbox{$\sum
_{j\in J}$} c_jD_{g_j}+D_b+D^\rho$
as in Claim \ref{claim2}.
If $
D\in{\rm ad\,}\W$, then  we may assume
$
D=\mbox{$\sum
_{j\in\Z}$}a_{j\,}{\rm ad}_{L_{0,j}}$
 for which only finitely many of $a_i$ are nonzero, since  $
 D$ is of degree zero and
 $\oplus_{i\in\Z}
 \mathbb F M_{0,i}$ lies in the center of $\W$. Thus
  \begin{equation}\label{clai-3-0}D_{\phi}+\mbox{$\sum\limits_{j\in J}$} c_jD_{g_j}+D_b+D^\rho
  =\mbox{$\sum\limits_{j\in\Z}$}a_{j\,}{\rm ad}_{L_{0,j}}.\end{equation} Applying both sides of \eqref{clai-3-0} to $L_{\al,i},$ we have $$\phi(\al)L_\al t^i+\mbox{$\sum\limits_{j\in J}$}  c_j{g_j}(\al)M_\al t^i+L_\al\rho(t^i)=\mbox{$\sum\limits_{j\in\Z}$} \al a_jL_\al t^{i+j},$$ which forces
$\mbox{$\sum
_{j\in J}$}c_jD_{g_j}=0$ and $\phi(\al)+t^{-i}\rho(t^i)=\mbox{$\sum
_{j\in\Z}$}\al a_jt^j$. The latter implies $\rho=0$ and thereby
$
\phi(\al)=\mbox{$\sum
_{j\in\Z}$}\al a_jt^j$ for $\al\in\Gamma.$ 
It follows from  applying both sides of \eqref{clai-3-0} to $M_{\al,i}$ and $Y_{\al+s,i}$  that we respectively  obtain \begin{eqnarray*}
&&\phi(\al)+b=\mbox{$\sum\limits_{j\in\Z}$}\al a_jt^j,
\ \ \ \ 
\phi(\al+s)+\frac12 b=\mbox{$\sum\limits_{j\in\Z}$}(\al+s) a_jt^j,\quad \forall\ \al\in\Gamma, i\in\Z.
\end{eqnarray*}Combining these three relations gives $b=0$ and $\phi(\gamma)=\gamma f(t)$ for any $\gamma\in T$, where
$f(t)=\mbox{$\sum
_{j\in\Z}$} a_jt^j\in\mathbb F[t,t^{-1}].$  To sum up, we have obtained $D=D_\phi=\mbox{$\sum
_{j\in\Z}$}a_j{\rm ad}_{L_{0,j}}$ with $\phi(\gamma)=\gamma f(t)$ for any $\gamma\in T.$ This is the Claim \ref{c-l-a-i3}.
\end{proof}

\section{Automorphism group of $\W$}
In this section, we will determine the automorphism group of $\W$,
which is denoted by ${\rm Aut\,}\W$. Denote
\begin{eqnarray*}&&
\mathcal{L}(\Gamma)=
\raisebox{-5pt}{${}^{\, \, \displaystyle\oplus}_{\alpha\in\Gamma}$}
L_\alpha\otimes \mathbb F[t,t^{-1}],\ \mathcal{M}(\Gamma)=
\raisebox{-5pt}{${}^{\, \, \displaystyle\oplus}_{\alpha\in\Gamma}$}
M_\alpha\otimes \mathbb F[t,t^{-1}],\ \mathcal{Y}(\Gamma,s)= \raisebox{-5pt}{${}^{\, \, \displaystyle\oplus}_{\alpha\in\Gamma}$}
Y_{\alpha+s}\otimes \mathbb F[t,t^{-1}].
\end{eqnarray*}
Obviously, $\mathcal{L}(\Gamma)$ is the centerless generalized loop Virasoro
algebra, and  $\mathcal{M}(\Gamma)$ is an ideal of $\mathcal{L}(\Gamma).$ It is also worthwhile to point out that
 $\mathcal{I}(\Gamma,s):=\mathcal{M}(\Gamma)\oplus \mathcal{Y}(\Gamma,s)$ is the unique maximal ideal of
$\W$ whose center is exactly $\mathcal M(\Gamma)$.

Denote  $\mathbb{F}^*=\mathbb F\setminus\{0\}.$ Set $A=\{a\in
\mathbb{F}^\ast\,|\,a{\Gamma}={\Gamma},aT=T\}$.  Then $A$ is a
multiplicative subgroup of $\mathbb{F}^*$.

\begin{theo}\label{theo3.2}Let
$\Gamma^\prime$ be another proper additive subgroup of $\F$ and $s'\notin\Gamma'$ with $2s'\in\Gamma'$. Set $T^\prime=\Gamma^\prime\cup(\Gamma^\prime+s^\prime).$
Then $\W\cong \mathscr W(\Gamma^\prime,s^\prime)$ if and only if there exists $a\in
\mathbb{F}^*$ such that $a\Gamma^\prime=\Gamma$ and  $a\Gamma_1^\prime=\Gamma_1$.
\end{theo}
\begin{proof} Let  $\sigma: \W\rightarrow \mathscr
W(T^\prime)$ be an isomorphism of Lie algebras.  Then
$\sigma\big(\mathcal{I}(\Gamma,s)\big)=\mathcal{I}(\Gamma^\prime,s^\prime)$, and $\sigma$ induces an isomorphism
$\bar\sigma$ of the generalized loop Witt algebras:
\begin{eqnarray*}
\overline{\sigma}:\mathcal L(\Gamma)=\W/\mathcal{I}(\Gamma,s)\rightarrow \mathscr
W(T^\prime)/\mathcal{I}(\Gamma^\prime,s^\prime)=\mathcal L(\Gamma^\prime).
\end{eqnarray*}
It follows from the arguments in  \cite[Section 3]{WWY} that there exists $a\in \mathbb{F}^*$
such that $a\Gamma^\prime=\Gamma,$ and $
\overline{\sigma}(L_{\al,i})=L_{\frac{\al}{a}}\,f_{\al,i}(t)
$ 
for some $f_{\al,i}(t)\in\mathbb{F}[t,t^{-1}]$ such that $f_{0,0}(t)=a$.
\setcounter{clai}{0}
\begin{clai}\label{3clai1}
$a\Gamma^\prime=\Gamma.$
\end{clai}
To prove this, for any $\al\in\Gamma$ and $i\in\Z,$ we can assume that
\begin{equation*}
\sigma(L_{\al,i})=L_{\frac{\al}{a}}\,f_{\al,i}(t)+m_{\alpha, i}+y_{\alpha, i}\quad {\rm for\ some\ }m_{\alpha, i}\in\mathcal M(\Gamma^\prime),
y_{\alpha, i}\in\mathcal Y(\Gamma^\prime,s^\prime).
\end{equation*}Note that $\sigma$ maps the center of $\mathcal I(\Gamma,s)$ onto the center of $\mathcal I(\Gamma^\prime,s^\prime)$, namely, $\sigma \big(\mathcal M(\Gamma)\big)=\mathcal M(\Gamma^\prime). $
It follows from this fact and by applying $\sigma$ to $[L_{0,0},M_{\al,i}]=\al M_{\al,i}$, we have
$$[L_{0,0},\sigma(M_{\al,i})]=\frac{\al}{a}
\sigma(M_{\al,i})\mbox{ for any $\alpha\in\Gamma$ and $i\in\Z$,}$$ which implies
$\sigma(M_{\al,i})\in {\mathscr W(T^\prime)}_\frac{\al}{a}$. In particular, $\frac{\al}{a}\in\Gamma^\prime$ and therefore
 $\Gamma\subseteq a\Gamma^\prime$. Similarly, one has
$a{\Gamma^\prime}\subseteq{\Gamma}$. Hence, $a\Gamma^\prime=\Gamma,$ proving Claim \ref{3clai1}.

\begin{clai}\label{3clai2}
$a{\Gamma^\prime_1}={\Gamma_1}.$
\end{clai}
To prove this, for any $\alpha\in\Gamma$, $i\in\Z$, since
$\sigma(\mathcal{Y}(\Gamma,s))=\mathcal{M}(\Gamma^\prime)\oplus\mathcal{Y}(\Gamma^\prime,s^\prime),$
we may write \begin{equation}\label{eq-3.1}
\sigma(Y_{\alpha+s,i})=\sigma(Y_{\alpha+s,i})_{\mathcal M}+\sigma(Y_{\alpha+s,i})_{\mathcal Y},
\end{equation}
for some $\sigma(Y_{\alpha+s,i})_{\mathcal M}\in\mathcal M(\Gamma^\prime),$
$  \sigma(Y_{\alpha+s,i})_{\mathcal Y}\in\mathcal Y(\Gamma^\prime,s^\prime).$
Applying $\sigma$ to $[L_{0,0},Y_{\alpha+s,i}]=$ \linebreak $(\alpha+s)
Y_{\alpha+s,i},$ one has $$
\big[aL_{0,0}+m_{0,0}+y_{0,0},\sigma(Y_{\alpha+s,i})\big]=(\alpha+s)\sigma(Y_{\alpha+s,i}),\quad \forall\ \alpha\in\Gamma, i\in\Z.$$
 Then it follows from relations
\eqref{rel3}--\eqref{rel6} and invoking \eqref{eq-3.1} that
$\big[L_{0,0}, \sigma(Y_{\alpha+s,i}\big)_{\mathcal
Y}\big]=$\linebreak $\frac{\al+s}{a}\sigma\big(Y_{\alpha+s,i}\big)_{\mathcal
Y},$
 which forces
$\frac{\al+s}{a}\in\Gamma^\prime_1$ and therefore $\Gamma_1\subset
a\Gamma^\prime_1$.  Then $a{\Gamma^\prime_1}={\Gamma_1}$,  since the other inclusion $a\Gamma^\prime_1\subset\Gamma_1$ can be obtained similarly.
 This proves Claim \ref{3clai2}.

The ``only if\," part follows directly from Claims \ref{3clai1} and \ref{3clai2}, while the ``if\," part is clear.\end{proof}

For any $\sigma\in {\rm Aut\,}\W,$ let us first take a careful look at what the images of the basis elements of $\W$ under the $\sigma$.   It follows from the proof of Theorem \ref{theo3.2} for any $\alpha\in \Gamma$ and $
i\in\Z,$ we may assume  that
\begin{eqnarray}\label{equation3.2}\!\!\!\!\!\!\!\!\!\!
&\!\!\!\!\!\!\!\!&\sigma(L_{\al,i})=L_{\frac{\al}{a}}\,f_{\al,i}+m^L_{\alpha,i}+y^L_{\alpha,i},
\ \
\sigma(M_{\al,i})=M_{\frac{\al}{a}}\,g_{\al,i},
\ \
\label{equation3.4}
\sigma(Y_{\al+s,i})=Y_{\frac{\al+s}{a}}\,g_{\al+s,i}+m^Y_{\alpha,i},
\end{eqnarray}
where    $m^L_{\alpha,i}, m^Y_{\alpha,i}\in \mathcal{M}(\Gamma),
y^L_{\alpha,i}\in \mathcal{Y}(\Gamma,s)$ and $f_{\al,i},g_{\al,i},g_{\al+s,i}\in \F[t,t^{-1}]$ with $f_{0,0}=a\in\mathbb F^*$. By \eqref{equation3.2}
and applying
$\sigma$ to the first relation in \eqref{rel1}, \eqref{rel3} and \eqref{rel6} respectively,  we have
\begin{eqnarray*}
&&a(\beta-\alpha)f_{\al+\beta,i+j}=(\beta-\alpha)f_{\al,i}
f_{\beta,j},\\
&&a(\beta+s-\frac{\alpha}{2})g_{\al+\beta+s,i+j}=(\beta+s-\frac{\alpha}{2})f_{\al,i}
g_{\beta+s,j},
\\
&&a(\beta-\alpha)g_{\al+\beta+2s,i+j}=(\beta-\alpha)g_{\al+s,i}
g_{\beta+s,j},\quad \forall\ \alpha,\beta \in \Gamma, i, j\in
\mathbb{Z}.
\end{eqnarray*}
Furthermore, by using the same arguments as in the proof of  \cite[Lemma
2.4]{WWY}, we have
\begin{eqnarray}\label{L-1}
&af_{\al+\beta,i+j}=f_{\al,i}f_{\beta,j},
\ \ \ \
ag_{\al+\beta+s,i+j}=f_{\al,i} g_{\beta+s,j},
\ \ \ \
ag_{\al+\beta+2s,i+j}=g_{\al+s,i}g_{\beta+s,j}
\end{eqnarray}
for any $\alpha,\beta \in \Gamma$ and $ i, j\in \mathbb{Z}.$
 Denote by  $\F[t,t^{-1}]^*=\{at^i\,|\,0\neq a\in\mathbb{F},\,i\in\Z\}$ the set of multiplicative invertible
elements of $\mathbb{F}[t,t^{-1}]$, which
 is a multiplicative group. It follows from the first and third relations of
\eqref{L-1} 
and by noticing both  $f_{0,0}$ and $g_{0,0}$ are nonzero elements of $\mathbb F$ that both $\frac{f_{\al,i}}{a}$ and $\frac{g_{\al+s,i}}{a}$ lie in $\F[t,t^{-1}]^*.$ Hence, by  the third formula in \eqref{L-1} again,  $g_{\alpha,i}\in\F[t,t^{-1}]^*,$ for any $\alpha\in\Gamma$ and $i\in\Z$.
These  entail us to write
\begin{equation*}
f_{\al,i}={a}\mu_1(\al,i)t^{\varepsilon_1(\al,i)}\quad {\rm and
}\quad
g_{\gamma,i}=a\mu^\prime(\gamma,i)t^{\varepsilon^\prime(\gamma,i)},\quad \forall\ \al\in\Gamma, \gamma\in T, i\in\Z,
\end{equation*}
where $\mu_1(\al,i),\mu^\prime(\gamma,i)\in\mathbb{F}^*,$
$\varepsilon_1(\al,i),\varepsilon^\prime(\gamma,i)\in\Z$.
Applying $\sigma$ to the second relation in \eqref{rel1}, we have $a\beta
g_{\al+\beta,i+j}=\beta f_{\al,i}g_{\beta,j}.$ From this and the second and third relations in \eqref{L-1}, 
one can deduce  $a
g_{\al+\beta,i+j}= f_{\al,i} g_{\beta,j}$ for $
\alpha,\beta\in\Gamma,i,j\in\Z,$ which implies
\begin{equation*}
\mu^\prime(\al,i)=c\mu_1(\al,i)\ (c=a^{-1}g_{0,0})\quad {\rm and}\quad
\varepsilon^\prime(\al,i)=\varepsilon_1(\alpha,i),  \quad \forall\ \al\in\Gamma, i\in\Z
\end{equation*}by taking $\beta=j=0.$
Set
\begin{equation}\label{definedn3.11}
\mu(x, i)=\left\{\begin{array}{llll}\mu_1(x,i),&\mbox{if \ }x\in \Gamma,\\[4pt]
c^{-\frac{1}{2}}\mu^\prime(x,i),&\mbox{if \ }x\in \Gamma_1,
\end{array}\right.
\mbox{ 
and
\ }
\varepsilon(x,i)=\left\{\begin{array}{llll}\varepsilon_1(x,i),&\mbox{if \ }x\in \Gamma,\\[4pt]
\varepsilon^\prime(x,i),&\mbox{if \ }x\in \Gamma_1.
\end{array}\right.\end{equation} Then for any $\gamma_1,\gamma_2\in T$ and $i,j\in\Z$, \begin{equation}\label{muand-varepsilon}\mu(\gamma_1,i)\mu(\gamma_2,j)=\mu(\gamma_1+\gamma_2,i+j)\quad {\rm and} \quad\varepsilon(\gamma_1,i)+\varepsilon(\gamma_2,j)=\varepsilon(\gamma_1+\gamma_2,i+j).\end{equation}
Whence the formulas in \eqref{equation3.2} 
 can be respectively rewritten as
\begin{eqnarray}\label{definedn3.13}
&&\sigma(L_{\al,i})={a}\mu(\al,i)L_\frac{\al}{a}t^{\varepsilon(\al,i)}+m^L_{\alpha,i}+y^L_{\alpha,i},
\ \ \ \
\sigma(M_{\al,i})={ac}\mu(\al,i)M_\frac{\al}{a}t^{\varepsilon(\al,i)},\\
\label{definedn3.15}
&&\sigma(Y_{\al+s,i})={ac^{\frac{1}{2}}}\mu(\al+s,i)Y_\frac{\al+s}{a}t^{\varepsilon(\al+s,i)}+m^Y_{\alpha,i}.
\end{eqnarray}

Motivated by the formulas \eqref{definedn3.13} and \eqref{definedn3.15} we are next going to define six kinds of automorphisms of $\W$.
For any $a\in A,$  define an automorphism of $\W$ as follows:
\begin{equation*}
\sigma_a(L_{\al,i})=aL_{\frac{\al}{a},i}, \quad
\sigma_a(M_{\al,i})=aM_{\frac{\al}{a},i}, \quad
\sigma_a(Y_{\al+s,i})=aY_{\frac{\al+s}{a},i},\quad \forall\ \al\in\Gamma,i\in\Z.
\end{equation*}

Given  any $\varphi\in {\rm {\rm Hom}}_\Z(T,\Z)$, the set of group homomorphisms
from $T$ to $\Z$,  we
can define an automorphism of $\W$ as follows:
\begin{equation*}
\varphi(L_{\al,i})=L_{\al,i+\varphi(\al)}, \quad
\varphi(M_{\al,i})=M_{\al,i+\varphi(\al)}, \quad
\varphi(Y_{\al+s,i})=Y_{\al+s,i+\varphi(\al+s)},\quad \forall\ \al\in\Gamma,i\in\Z.
\end{equation*}

 For  any $\chi\in\chi(T)$, the set of all group homomorphisms from $T$ to $\mathbb F^*$,   and $c\in\mathbb F^*$, define an automorphism $\sigma^{\chi}_c$
of $\W$ as follows:
\begin{equation*}
\sigma^{\chi}_c(L_{\al,i})=\chi(\al)L_{\al,i}, \
\sigma^{\chi}_c(M_{\al,i})=c\chi(\al)M_{\al,i}, \
\sigma^{\chi}_c(Y_{\al+s,i})=c^{\frac{1}{2}}\chi(\al+s)Y_{\al+s,i},\,\ \forall\ \al\in\Gamma,i\in\Z.
\end{equation*}
Then under the multiplication  given by
$\sigma^{\chi_1}_{c_1}\sigma^{\chi_2}_{c_2}=\sigma^{\chi_1\chi_2}_{c_1c_2}$
for all $c_1,c_2\in \F^*$ and $\chi_1,\chi_2\in\chi(T),$  the set $\{\sigma^{\chi}_c\mid \chi\in \chi(T), c\in\F^*  \}$ forms a subgroup of ${\rm Aut\,} \W$, which is isomorphic to the product group  $\chi(T)\times \F^*$ of groups $\chi(T)$ and $\F^*.$

For any $\phi\in {\rm Aut\,}\Z=\{1,-1\}$, it is easy to see the following linear map given by
\begin{equation*}
{\phi}(L_{\al,i})=L_{\al,\phi(i)}, \quad
{\phi}(M_{\al,i})=M_{\al,\phi(i)}, \quad
{\phi}(Y_{\al+s,i})=Y_{\al+s,\phi(i)},\quad\forall\ \alpha\in \Gamma, i\in \Z
\end{equation*}
is an automorphism of $\W$.

Given an element $b\in \mathbb{F}^*$,  define a linear map $\tau_b\in \F^\times$ as follows:
\begin{equation*}
\tau_{b}(L_{\al,i})=b^iL_{\al,i},\quad
\tau_{b}(M_{\al,i})=b^iM_{\al,i}, \quad
\tau_{b}(Y_{\al+s,i})=b^iY_{\al+s,i}, \quad \forall\ \alpha\in \Gamma,i\in \Z.
\end{equation*}
Then $\tau_b$ is an automorphism of $\W$.

Note from \cite{WWY} that  the set $A\times{{\rm Hom}_\Z(T,\Z)}\times\chi(T)\times\F^*\times{{\rm Aut\,}\Z}\times\F^\times$ becomes a group with  $(1,0,0,1,1,1)$ as its identity element,  whose multiplication is given as follows:
\begin{eqnarray*}\label{eqwww3.1}
&&(a_1,\varphi_1,\chi_1,c_1,\phi_1,b_1)\cdot(a_2,\varphi_2,\chi_2,c_2,\phi_2,b_2) \nonumber\\
&&=(a_1a_2,\varphi_1\iota_{a_2^{-1}}+\phi_1\varphi_2,b_1^{\varphi_2}+\chi_2+\chi_1\iota_{a_2^{-1}},c_1c_2,\phi_1\phi_2,b_1^{\phi_2}b_2),
\end{eqnarray*}
where  $\iota_a$ (for $a\in A$)   is an automorphism of $T$ such that
$
\iota_a(\gamma)=a\gamma\mbox{ for }
\gamma\in T,
$ 
and  for any $b\in \F^*$ and $\varphi\in$ ${\rm Hom}_\Z(T,\Z)$,
$
(b^{\varphi})(\gamma)=b^{\varphi(\gamma)}\mbox{ for }
\gamma\in T.
$ 
So in what follows we can  identify  $A\times{{\rm Hom}_\Z(T,\Z)}\times\chi(T)\times\F^*\times{{\rm Aut\,}\Z}\times\F^\times$ as a subgroup of ${\rm Aut}\ \W$.

Denote
\begin{eqnarray*}
 \mathcal{E}:=
\left\{\underline e=(e^k)_{k\in\Z}\ \bigg|\ \begin{aligned} &{\rm all\ but\ finitely\ many}\ e^k\ {\rm are\ zero, where}\ e^k: \Gamma\times \Z\rightarrow \mathbb F\ {\rm such} \\[-5pt] &{\rm that}\  (\beta-\alpha)e_{\alpha+\beta,i+j}^k=\beta e_{\beta,j}^{k-i} -\alpha e_{\alpha,i}^{k-j},\ \forall\, \alpha,\beta\in\Gamma,
i,j\in\Z\end{aligned}\right\}.
\end{eqnarray*}
Then one can verify that the linear map $\psi_{\underline e}:\W\rightarrow \W$ defined by
\begin{equation}\label{psi_{e}}
L_{\alpha,i}\rightarrow L_{\alpha,i}+\mbox{$\sum\limits_{j\in\Z}$}e_{\alpha,i}^j M_{\alpha,j}, \quad
M_{\alpha,i} \rightarrow M_{\alpha,i}, \quad Y_{\alpha+s,i}
\rightarrow Y_{\alpha+s,i}
\end{equation}
is an automorphism of $\W$. Note that $\mathcal{E}$ carries the semigroup structure under the natural additive operation $\underline e+\underline c=(e^k+c^k)_{k\in\Z}$ for any $\underline e=(e^k)_{k\in\Z},\underline c=(c^k)_{k\in\Z}\in\mathcal E$. Thus,  $\Psi=\{\psi_{\underline e}\mid \underline e\in
\mathcal{E}\}$ forms a subgroup of ${\rm Aut\,}\W$, since
$\psi_{\underline e} \psi_{\underline c} =\psi_{\underline e+\underline c}$ for any
$\underline e,\underline c\in \mathcal{E}.$

Let ${\rm Inn\,}\W$ be the subgroup of  ${\rm Aut\,}\W$ generated by all {\em inner automorphisms} ${\rm exp\,} ({\rm ad} x)$ for all $x\in\W$ such that each ${\rm ad }x$ is locally nilpotent.
 It is not hard to see that all automorphisms defined above are not inner.  Now we are ready to characterize the automorphism group ${\rm Aut}\, \W$.

\begin{theo}\label{theo3.4}We have the following isomorphism
$$
{\rm Aut\,}\W\simeq((A\times{\rm Hom}_\Z(T,\Z)\times\chi(T)\times \F^*\times{{\rm
Aut\,}\Z}\times\mathbb{F}^\times)\ltimes\Psi)\ltimes  {\rm Inn\,}\W.
$$
\end{theo}

\begin{proof} Let $\sigma\in {\rm Aut \,}\W$. Then it follows from \eqref{muand-varepsilon} that the $\mu$ and $\varepsilon$ occur in \eqref{definedn3.13} and \eqref{definedn3.15} can be written as $$\mu(\gamma,i)=b^i\chi(\gamma)\quad {\rm and}\quad \varepsilon(\gamma,i)=\phi(i)+\varphi(\gamma)$$ for some $b\in\mathbb F^*$, $\chi\in\chi(T)$, $\phi\in{\rm Aut}\,\Z$ and $\varphi\in {\rm {\rm Hom}}_\Z(T,\Z).$
This 
entails us to rewrite \eqref{definedn3.13} and \eqref{definedn3.15} respectively as follows:
\begin{eqnarray}\label{eqna3.9}
&&\sigma(L_{\alpha,i})={a}{\chi(\al)}{b}^{i}L_{{\frac{\al}{a}},{\phi(i)+\varphi(\al)}}+m^L_{\alpha,i}+y^L_{\alpha,i},\\
\label{eqna3.10}
&&\sigma(M_{\alpha,i})=ac{\chi(\al)}{b}^{i}M_{{\frac{\al}{a}},{\phi(i)+\varphi(\al)}},\\
\label{eqna3.11}
&&\sigma(Y_{\alpha+s,i})=ac^{\frac{1}{2}}{\chi(\al+s)}{b}^{i}Y_{{\frac{\al+s}{a}},{\phi(i)+\varphi(\alpha+s)}}+m^Y_{\alpha,i},
\end{eqnarray}
where  $a\in A$, $c\in\F^*$, $m^L_{\alpha,i},
m^Y_{\alpha,i}\in \mathcal{M}(\Gamma)$ and $y^L_{\alpha,i}\in
\mathcal{Y}(\Gamma,s).$ Set
\begin{equation}\label{def-tau}\tau=(\tau_b)^{-1}\phi^{-1}(\sigma_{c}^\chi)^{-1}\varphi^{-1}(\sigma_a)^{-1}\sigma. \end{equation}   Then it follows from \eqref{eqna3.9}$-$\eqref{eqna3.11} and the definitions of $\tau_b, \phi, \sigma_{c}^\chi, \varphi$ and $\sigma_a$ that
\begin{eqnarray}\label{eq3.4}
&\!\!\!\!\!\!\!\!\!\!\!\!&\tau(L_{\alpha,i})=L_{\alpha,i}+\mathfrak{m}^L_{\alpha,i}+\mathfrak{y}^L_{\alpha,i},
\ \ \ \ \
\tau(M_{\alpha,i})=M_{\alpha,i},
\ \ \ \ \
\tau(Y_{\alpha+s,i})=Y_{\alpha+s,i}+\mathfrak{m}^Y_{\alpha,i},
\end{eqnarray}
where $
\mathfrak{m}^L_{\alpha,i},\mathfrak{m}^Y_{\alpha,i}\in \mathcal{M}(\Gamma)$ and
$\mathfrak{y}^L_{\alpha,i}\in \mathcal{Y}(\Gamma,s).$

Assume
\begin{eqnarray}\label{eqnarrayyy3.20}
&&\tau(L_{0,0})=L_{0,0}+\mbox{$\sum\limits_{\alpha\in\Gamma, j\in\Z}$}(a_{\alpha,j}M_{\alpha
,j}
+b_{\alpha,j}Y_{\alpha+s,j}),\\
&&\label{eqnarrayyy3.21}
\tau(Y_{s,i})=Y_{s,i}+\mbox{$\sum\limits_{\beta\in\Gamma, k\in\Z}$}c_{\beta,k}^iM_{\beta,k},
\end{eqnarray} for which only finitely many of $a_{\al,j}, b_{\al,j}$ and $c_{\beta,k}^i$ for all $\al,\beta\in\Gamma,i,j,k\in\Z$ are nonzero.
Applying $\tau$ to $[L_{0,0},Y_{s,i}]=s Y_{s,i}$, we get
\begin{eqnarray*}
&&\Big[L_{0,0}+\mbox{$\sum\limits_{\alpha\in\Gamma, j\in\Z}$}(a_{\alpha,j}M_{\alpha
,j}+b_{\alpha,j}Y_{\alpha+s,j}),Y_{s,i}+\mbox{$\sum\limits_{\beta\in\Gamma, k\in\Z}$}c_{\beta,k}^iM_{\beta,k}\Big]\\
&&=
s Y_{s,i}+\mbox{$\sum\limits_{\beta\in\Gamma, k\in\Z}$}\beta c_{\beta,k}^iM_{\beta,k}-\mbox{$\sum\limits_{\alpha\in\Gamma, j\in\Z}$}\alpha b_{\alpha,j}M_{\alpha+2s
,j+i}\\
&&= s Y_{s,i}+s\mbox{$\sum\limits_{\beta\in\Gamma, k\in\Z}$}c_{\beta,k}^iM_{\beta,k}.
\end{eqnarray*}
Thus we conclude
$
\mbox{$\sum
_{\beta\in\Gamma, k\in\Z}$}(\beta-s)c_{\beta,k}^iM_{\beta,k}=
\mbox{$\sum
_{\alpha\in \Gamma,j\in\Z}$}\alpha b_{\alpha,j}M_{\alpha+2s,j+i},
$ 
which is equivalent to saying
\begin{equation*}
c_{\alpha+2s,j+i}^i=\frac{\alpha}{\alpha+s}b_{\alpha,j},\quad\forall\ \al\in\Gamma,i,j\in\Z.
\end{equation*}
Then \eqref{eqnarrayyy3.21} can be written as
\begin{equation}\label{eq3.7}
\tau(Y_{s,i})=Y_{s,i}+\mbox{$\sum\limits_{\alpha\in\Gamma, j\in\Z}$}\frac{\alpha}{\alpha+s}b_{\alpha,j}M_{\alpha+2s,j+i}.
\end{equation}
 Let
\begin{eqnarray*}\label{eqqqqa3.23}
\tau^\prime&\!\!=&\!\!{\rm exp\, }{\rm ad}\Big(\mbox{$\sum\limits_{\alpha\neq\beta\in \Gamma, \al+\beta+2s\neq0,i,j\in\Z}$}\frac{-b_{\alpha,i}b_{\beta,j}(\beta-\alpha)}{2(\alpha+s)(\alpha+\beta+2s)}
M_{\alpha+\beta+2s,i+j}\Big)
\\
&&\!\!\times\,{\rm exp\,}{\rm ad}\Big(\mbox{$\sum\limits_{0\neq\alpha\in \Gamma,j\in\Z}$}\frac{-a_{\alpha,j}}{\alpha}M_{\alpha,j}\Big){\rm exp\, }{\rm
ad}\Big(-\mbox{$\sum\limits_{\alpha\in \Gamma,j\in\Z}$}\frac{b_{\alpha,j}}{\alpha+s}Y_{\alpha+s,j}\Big)\in {\rm Inn\ }\W.
\end{eqnarray*}
\begin{clai}\label{c-l-a}
We have $\tau^\prime(L_{\al,i})=\tau(L_{\al,i})+\mbox{$\sum
_{j\in\Z}$}e_{\al,i}^jM_{\al,j}$,  $\tau^\prime(Y_{\al+s,i})=\tau(Y_{\al+s,i})$ and $\tau^\prime(M_{\al,i})=M_{\al,i}=\tau(M_{\al,i})$ for any $\al\in\Gamma$ and $i\in\Z$, where $e_{\al,i}^j\in\mathbb F$ for which only finitely many are nonzero for each fixed $\al$.
\end{clai}
To prove the claim, it follows immediately from the construction of $\tau^\prime$ one can easily verify that  \begin{eqnarray}\label{eq-3.23}\!\!\!\!
\tau^\prime(L_{0,0})=\tau(L_{0,0})-\mbox{$\sum\limits_{j\in\Z}$}a^\prime_{0,j}M_{0,j},\ \tau^\prime(Y_{s,i})=\tau(Y_{s,i}),\ \tau^\prime(M_{\al,i})=M_{\al,i},\,\ \forall\ \al\in\Gamma,i\in\Z,
\end{eqnarray}where $a^\prime_{0,j}=a_{0,j}-\mbox{$\sum
_{\alpha\in\Gamma,k\in\Z}$}b_{\al,k}b_{-\al-2s,j-k}.$
Applying $\tau^\prime$ to
$[L_{0,0},L_{\alpha,i}]=\alpha L_{\alpha,i}$ for arbitrary  $\alpha\in \Gamma$ and $i\in\Z,$ we obtain
\begin{eqnarray*}\alpha\tau^\prime(L_{\alpha,i}) &=&\Big[\tau^\prime(L_{0,0}),\tau^\prime(L_{\alpha,i})\Big]=\Big[\tau(L_{0,0}),\tau^\prime(L_{\alpha,i})\Big]
\nonumber\\
&=&\Big[\tau(L_{0,0}),\tau(L_{\alpha,i})+
\tau^\prime(L_{\alpha,i})-\tau(L_{\alpha,i})\Big]
\nonumber\\
&=&\alpha
\tau(L_{\alpha,i})+\Big[\tau(L_{0,0}),\tau^\prime(L_{\alpha,i})-\tau(L_{\alpha,i})\Big],
\end{eqnarray*} where in the second equality we use the fact that $M_{0,j}$ for all $j\in\Z$ lie in the center of $\W$.
Thus\begin{equation}\label{eq3.8}
\big[\tau(L_{0,0}),\tau^\prime(L_{\alpha,i})-\tau(L_{\alpha,i})\big]=\alpha\big(\tau^\prime(L_{\alpha,i})-\tau(L_{\alpha,i})\big).
\end{equation}

Note that  $[\tau^\prime(L_{\alpha,i}),M_{\beta,j}]=\beta
M_{\alpha+\beta,i+j}$, which is due to the fact $\tau^\prime$ preserves each basis element $M_{\beta,j}$. Thus,
$\tau^\prime(L_{\alpha,i})-L_{\alpha,i}\in\mathcal{M}(\Gamma)\oplus \mathcal{Y}(\Gamma,s)$. This and  \eqref{eq3.4} allow us to assume that
\begin{equation}\label{equaeq3.24}
\tau^\prime(L_{\alpha,i})-\tau(L_{\alpha,i})=
\mbox{$\sum\limits_{\beta\in \Gamma,j\in\Z}$}(e_{\alpha,i}^{\beta,j}M_{\beta,j}+d_{\alpha,i}^{\beta,j}Y_{\beta+s,j})\end{equation} for some $e_{\al,i}^{\beta,j}, d_{\al,i}^{\beta,j}\in\mathbb F$ with $e_{0,0}^{0,j}=-a^\prime_{0,j}$ and $d_{0,0}^{\beta,j}=0$.
Inserting  \eqref{eqnarrayyy3.20} and  \eqref{equaeq3.24} into \eqref{eq3.8}  we have
\begin{eqnarray*}\\
&&\alpha(\mbox{$\sum\limits_{\beta\in \Gamma,j\in\Z}$}e_{\alpha,i}^{\beta,j}M_{\beta,j}
+\mbox{$\sum\limits_{\beta\in \Gamma,j\in\Z}$}d_{\alpha,i}^{\beta,j}Y_{\beta+s,j})\\
&&=\Big[L_{0,0}+\mbox{$\sum\limits_{\beta\in\Gamma, j\in\Z}$}(a_{\beta,j}M_{\beta
,j}+b_{\beta,j}Y_{\beta+s,j}),\mbox{$\sum\limits_{\gamma\in \Gamma,j\in\Z}$}(e_{\al,i}^{\gamma,j}M_{\gamma,j}
+d_{\al,i}^{\gamma,j}Y_{\gamma+s,j})\Big]\nonumber\\
&&=\mbox{$\sum\limits_{\gamma\in\Gamma, j\in\Z}$}\Big(\gamma e_{\alpha,i}^{\gamma,j}M_{\gamma
,j}+(\gamma+s)d_{\alpha,i}^{\gamma,j}Y_{\gamma+s,j}\Big)
+\mbox{$\sum\limits_{\beta\neq\gamma\in \Gamma,j,k\in\Z}$}(\gamma-\beta)b_{\beta,k}d_{\alpha,i}^{\gamma,j}M_{\beta+\gamma+2s,j+k}.
\end{eqnarray*}
Comparing coefficients of $Y_{\beta+s,j}$  gives
$(\beta+s)d_{\al,i}^{\beta,j}=\alpha d_{\al,i}^{\beta,j}.$
Hence, $d_{\al,i}^{\beta,j}=0$ for $\al,\beta\in \Gamma$ and $i,j\in\Z$
since $\beta+s-\al\neq0$. While comparing coefficients of $M_{\beta,j},$ we have
$ e_{\al,i}^{\beta,j}=0\mbox{ for }
\al\neq\beta\in\Gamma, i,j\in\Z.$
Consequently, \eqref{equaeq3.24} can be simply written as
\begin{equation}\label{eq3.9}\tau^\prime(L_{\alpha,i})=\tau(L_{\alpha,i})+\mbox{$\sum\limits_{j\in\Z}$} e_{\alpha,i}^j M_{\alpha,j}\quad {\rm with}\ e_{\alpha,i}^j=e_{\alpha,i}^{\al,j}.\end{equation}
This is the first relation as promised in Claim \ref{c-l-a}.

 It follows from \eqref{eq3.7}, \eqref{eq-3.23} and \eqref{eq3.9} that
\begin{eqnarray*}\tau^\prime(Y_{\alpha+s,i})&=&(s-\frac{\alpha}{2})^{-1}\tau^\prime[L_{\alpha,0},Y_{s,i}]
\nonumber\\
&=&(s-\frac{\alpha}{2})^{-1}\Big[\tau(L_{\alpha,0})+\mbox{$\sum\limits_{j\in\Z}$}e_{\alpha,0}^j M_{\alpha,j},\tau(Y_{s,i})\Big]
=
\tau(Y_{\alpha+s,i}),\quad\forall\ 2s\neq\alpha\in \Gamma.
\end{eqnarray*}
Using this result and \eqref{eq3.9}, we have
\begin{eqnarray*}\tau^\prime(Y_{3s,i})&=&(-3s)^{-1}\tau^\prime[L_{4s,0},Y_{-s,i}]\nonumber\\
&=&(-3s)^{-1}\Big[\tau(L_{4s,0})+\mbox{$\sum\limits_{j\in\Z}$}e_{4s,0}^jM_{4s,j},\tau(Y_{-s,i})\Big]
=
\tau(Y_{3s,i}).
\end{eqnarray*}
So in either case we have proved the fact that
$
 \tau^\prime(Y_{\alpha+s,i})=\tau(Y_{\alpha+s,i})\mbox{ for }
 \alpha\in \Gamma, i\in\Z.
$ 
Whence we have established Claim \ref{c-l-a}.

In fact, Claim \ref{c-l-a} implies $\tau=\tau^\prime\psi_{\underline e}^{-1}\in {\rm Inn\,}\W\cdot \Psi$, where  $\psi_{\underline e}\in {\rm Aut\,} \W$ as defined in \eqref{psi_{e}}.
This together with \eqref{def-tau} gives rise to
\begin{equation*}
\sigma=\sigma_a\varphi\sigma_c^{\chi}\phi\tau_{b}\tau^\prime\psi_{\underline e}^{-1}\in
(A\times{{\rm Hom}_\Z(T,\Z)}\times(\chi(T)\times
\F^*)\times{{\rm Aut\,}\Z}\times\mathbb{F}^\times)\cdot {\rm Inn\,}\W\cdot\Psi.
\end{equation*}
Since ${\rm Inn\,}\W$ is a normal subgroup of ${\rm Aut\,}\W$, thus
${\rm Inn\,}\W\cdot\Psi=\Psi\cdot {\rm Inn\,}\W.$

In order to finish the whole proof, it remains to show the following is a simiproduct: $$(A\times{\rm Hom}_\Z(T,\Z)\times\chi(T)\times \F^*\times{{\rm
Aut\,}\Z}\times\mathbb{F}^\times)\ltimes\Psi, $$ which is equivalent to proving the following claim.
\begin{clai}\label{cla4}
For any $\sigma_a\in A, \varphi \in {\rm Hom}_\Z(T,\Z), \sigma_c^\chi\in \chi(T)\times \F^*, \phi\in {{\rm
Aut\,}\Z}$, $\tau_b\in \mathbb{F}^\times$, and $\psi_{\underline e}$ as defined in \eqref{psi_{e}}, we have $$(\sigma_a\varphi\sigma_c^\chi\phi\tau_b)^{-1}\psi_{\underline e}(\sigma_a\varphi\sigma_c^\chi\phi\tau_b)\in\Psi.$$
\end{clai}
In fact, we are going to show $$(\sigma_a\varphi\sigma_c^\chi\phi\tau_b)^{-1}\psi_{\underline e}(\sigma_a\varphi\sigma_c^\chi\phi\tau_b)=\psi_{\underline d},$$ where $\underline d=(d^k)_{k\in\Z}\in \mathcal{E}$ with $d^k$ defined by $d^k_{\al,i}=c^{-1}b^{i-k}e_{\frac{\alpha}{a},\phi(i)+\varphi(\alpha)}^{\phi(k)+\varphi(\alpha)}$ for any $\al\in\Gamma$ and $i\in\Z$.
It follows from the definitions of $\sigma_a,\varphi,\sigma_c^\chi,\phi,\tau_b$ and $\psi_{\underline{e}}$ that we obtain
\begin{eqnarray*}
&&(\sigma_a\varphi\sigma_c^\chi\phi\tau_b)^{-1}\psi_{\underline e}(\sigma_a\varphi\sigma_c^\chi\phi\tau_b)(L_{\alpha,i})
=(\sigma_a\varphi\sigma_c^\chi\phi\tau_b)^{-1}\psi_{\underline e}(\sigma_a\varphi\sigma_c^\chi\phi)
(b^iL_{\alpha,i})
\\
&&=(\sigma_a\varphi\sigma_c^\chi\phi\tau_b)^{-1}\psi_{\underline e}(\sigma_a\varphi\sigma_c^\chi)
(b^iL_{\alpha,\phi(i)})=(\sigma_a\varphi\sigma_c^\chi\phi\tau_b)^{-1}\psi_{\underline e}(\sigma_a\varphi)
(b^i\chi(\alpha)L_{\alpha,\phi(i)})
\\
&&=(\sigma_a\varphi\sigma_c^\chi\phi\tau_b)^{-1}\psi_{\underline e}(\sigma_a)
(b^i\chi(\alpha)L_{\alpha,\phi(i)+\varphi(\alpha)})
=(\sigma_a\varphi\sigma_c^\chi\phi\tau_b)^{-1}\psi_{\underline e}
(b^i\chi(\alpha)aL_{\frac{\alpha}{a},\phi(i)+\varphi(\alpha)})
\\
&&=\tau_b^{-1}\phi^{-1}(\sigma_c^\chi)^{-1}\varphi^{-1}(\sigma_a)^{-1}
(b^i\chi(\alpha)aL_{\frac{\alpha}{a},\phi(i)+\varphi(\alpha)}+b^i\chi(\alpha)a\mbox{$\sum\limits_{j\in\Z}$}e_{\frac{\alpha}{a},\phi(i)+\varphi(\alpha)}^{\phi(j)+\varphi(\alpha)}
M_{\frac{\alpha}{a},\phi(j)+\varphi(\alpha)})
\\
&&=\tau_b^{-1}\phi^{-1}(\sigma_c^\chi)^{-1}\varphi^{-1}
(b^i\chi(\alpha)L_{\alpha,\phi(i)+\varphi(\alpha)}+b^i\chi(\alpha)\mbox{$\sum\limits_{j\in\Z}$}e_{\frac{\alpha}{a},\phi(i)+\varphi(\alpha)}^{\phi(j)+\varphi(\alpha)}
M_{\alpha,\phi(j)+\varphi(\alpha)})
\\
&&=\tau_b^{-1}\phi^{-1}(\sigma_c^\chi)^{-1}
(b^i\chi(\alpha)L_{\alpha,\phi(i)}+b^i\chi(\alpha)\mbox{$\sum\limits_{j\in\Z}$}e_{\frac{\alpha}{a},\phi(i)+\varphi(\alpha)}^{\phi(j)+\varphi(\alpha)}
M_{\alpha,\phi(j)})
\\
&&=\tau_b^{-1}\phi^{-1}
(b^iL_{\alpha,\phi(i)}+\mbox{$\sum\limits_{j\in\Z}$}c^{-1}b^ie_{\frac{\alpha}{a},\phi(i)+\varphi(\alpha)}^{\phi(j)+\varphi(\alpha)}
M_{\alpha,\phi(j)})
=\tau_b^{-1}
(b^iL_{\alpha,i}+\mbox{$\sum\limits_{j\in\Z}$}c^{-1}b^ie_{\frac{\alpha}{a},\phi(i)+\varphi(\alpha)}^{\phi(j)+\varphi(\alpha)}
M_{\alpha,j})\\
&&=
L_{\alpha,i}+\mbox{$\sum\limits_{j\in\Z}$}c^{-1}b^{i-j}e_{\frac{\alpha}{a},\phi(i)+\varphi(\alpha)}^{\phi(j)+\varphi(\alpha)}
M_{\alpha,j}=\psi_{\underline d}(L_{\al,i}).
\end{eqnarray*}
While for $M_{\al, i}$ and $Y_{\al+s,i}$, we have
\begin{eqnarray*}&&
(\sigma_a\varphi\sigma_c^\chi\phi\tau_b)^{-1}\psi_{\underline e}(\sigma_a\varphi\sigma_c^\chi\phi\tau_b)(M_{\al,i})=M_{\al,i}=\psi_{\underline d}(M_{\al,i}),
\\
&& (\sigma_a\varphi\sigma_c^\chi\phi\tau_b)^{-1}\psi_{\underline e}(\sigma_a\varphi\sigma_c^\chi\phi\tau_b)(Y_{\al+s,i})=Y_{\al+s,i}=\psi_{\underline d}(Y_{\al+s,i}).
\end{eqnarray*}Thus, $(\sigma_a\varphi\sigma_c^\chi\phi\tau_b)^{-1}\psi_{\underline e}(\sigma_a\varphi\sigma_c^\chi\phi\tau_b)=\psi_{\underline d},$ as desired.
\end{proof}

\section{Second cohomology group of $\W$}

 We start with some relevant notions. A bilinear form
$\W\times \W \rightarrow \mathbb{F}$ is called a {\it 2-cocycle}
on $\W$ if the following conditions are satisfied:
\begin{eqnarray*}
&&\psi(x,y)=-\psi(y,x),\\
&&\psi(x,[y,z])+\psi(y,[z,x])+\psi(z,[x,y])=0, \quad \forall\
x,y,z \in \W.
\end{eqnarray*}
 Denote by $C^{2}(\W,\ \mathbb{F})$ the vector space of
2-cocycles on
 $\W$. For any linear function $f: \W \rightarrow
 \mathbb{F}$, one can define a 2-cocycle $\psi_{f}$ as follows:
\begin{eqnarray}\label{eq4.1}
&&\psi_{f}(x,y)=f([x,y]), \quad\forall\ x,y \in \W.
\end{eqnarray}
Such a 2-cocycle is called a {\it 2-coboundary} on $\W$. Denote by
$B^{2}(\W,\ \mathbb{F})$ the vector space of 2-coboundaries on
 $\W$. The quotient space
 \begin{eqnarray*}
&&H^{2}(\W,\ \mathbb{F})=C^{2}(\W,\ \mathbb{F})/B^{2}(\W,\
\mathbb{F})
\end{eqnarray*}
is called  the {\it second cohomology group} of $\W$. The 2-cocyles
are closely related to  central extensions of
Lie algebras. To be more precise,  there exists a one-to-one correspondence between the
set of equivalence classes of one-dimensional central extensions of
$\W$ by $\mathbb{F}$ and the second cohomology group of $\W$.

Now let $\phi$ be a 2-cocycle of $\W$. Define an
$\mathbb{F}$-linear map $f:\W\rightarrow \mathbb{F}$ as follows:
\begin{eqnarray}\label{definedn4.1}&&
f(L_{\alpha,i})=\left\{\begin{array}{llll}\frac{1}{\alpha}\phi(L_{0,0},L_{\alpha,i})&\mbox{if \ }\alpha\neq 0,\\[4pt]
-\frac{1}{4s}\phi(L_{2s,0},L_{-2s,i})&\mbox{if \ }\alpha =0,
\end{array}\right.\\
\label{definedn4.2}
&&f(M_{\alpha,i})=\left\{\begin{array}{llll}\frac{1}{\alpha}\phi(L_{0,0},M_{\alpha,i})&\mbox{if \ }\alpha\neq 0,\\[4pt]
-\frac{1}{2s}\phi(L_{2s,0},M_{-2s,i})&\mbox{if \ }\alpha =0,
\end{array}\right.\\
\label{definedn4.3}
&&f(Y_{\alpha+s,i})=\frac{1}{\alpha+s}\phi(L_{0,0},Y_{\alpha+s,i}),
 \quad\quad \forall\ \,\alpha\in \Gamma.
\end{eqnarray}Set $\varphi=\phi-\psi_f$, where $\psi_f$ is defined in
\eqref{eq4.1}. Using the Jacobi identity on the triples
$(L_{\alpha,i},L_{\beta,j},M_{\gamma,k})$ and
$(L_{\alpha,i},Y_{\beta+s,j},Y_{\gamma+s,k})$, we respectively get
\begin{eqnarray}\label{eq4.5}
\gamma\varphi(L_{\alpha,i},M_{\beta+\gamma,j+k})
&=&\varphi(L_{\alpha,i}, [L_{\beta,j},M_{\gamma,k}])\nonumber\\
&=&\varphi\big([L_{\alpha,i},L_{\beta,j}],M_{\gamma,k}\big)+\varphi\big(L_{\beta,j}, [L_{\alpha,i},M_{\gamma,k}]\big)\nonumber\\
&=&(\beta-\alpha)\varphi(L_{\alpha+\beta,i+j},M_{\gamma,k})+\gamma\varphi(L_{\beta,j},M_{\alpha+\gamma,i+k})
\end{eqnarray}
and
\begin{eqnarray}\label{eq4.6}
&&(\gamma-\beta)\varphi(L_{\alpha,i},M_{\beta+\gamma+2s,j+k})
=
\varphi\big(L_{\alpha,i}, [Y_{\beta+s,j},Y_{\gamma+s,k}]\big)\nonumber\\
&&=\varphi\big([L_{\alpha,i},Y_{\beta+s,j}],Y_{\gamma+s,k}\big)+\varphi\big(Y_{\beta+s,j}, [L_{\alpha,i},Y_{\gamma+s,k}]\big)\nonumber\\
&&=(\beta+s-\frac{\alpha}{2})\varphi(Y_{\alpha+\beta+s,i+j},Y_{\gamma+s,k})
+(\gamma+s-\frac{\alpha}{2})\varphi(Y_{\beta+s,j},Y_{\alpha+\gamma+s,i+k})
\end{eqnarray}
for all $\alpha,\beta,\gamma\in \Gamma$ and $i,j,k\in\Z$. Setting $\al=i=0$ and $\gamma=-\beta\neq0$ in \eqref{eq4.5}, one has $\varphi(L_{0,0}, M_{0,j})=0$ for any $j\in\Z$.
This together with  \eqref{definedn4.2} and \eqref{definedn4.3} gives
\begin{equation}\label{definedn13}
\varphi(L_{0,0},M_{\alpha,i})=\varphi(L_{2s,0},M_{-2s,i})=\varphi(L_{0,0},Y_{\al+s,i})=0
\mbox{
for all $\alpha\in \Gamma$ and $i\in\Z$.
}\end{equation}
For simplicity, denote
\begin{eqnarray*}
&&A_{\alpha,\beta,i,j}=\varphi(L_{\alpha,i},M_{\beta,j}),\quad\ \ \ \ \ \ \ \ \ \
A_{\alpha,i,j}=\varphi(L_{\alpha,i},M_{-\alpha,j}),\\
&&B_{\alpha+s,\beta+s,i,j}=\varphi(Y_{\alpha+s,i},Y_{\beta+s,j}),\quad \ B_{\alpha+s,i,j}=\varphi(Y_{\alpha+s,i},Y_{-\alpha-s,j})
\end{eqnarray*}
for all $\alpha,\beta\in \Gamma$ and $i,j\in\Z$.

We are next   devoted to deriving that $A_{\al,\beta,i,j}=B_{\al,\beta,i,j}=0$ for any $\al,\beta\in\Gamma$ and $i,j\in\Z$. To do this,   it first follows
by setting $\gamma=-2\alpha, \beta=\alpha$ in \eqref{eq4.5} that
$\alpha A_{\alpha,i,j+k}=\alpha A_{\alpha,j,i+k}.$
Hence,
\begin{equation}\label{rel-ab-A}
 A_{\alpha,i,j+k}=A_{\alpha,j,i+k}, \quad \forall\ \alpha\neq 0,i,j,k \in
 \mathbb{Z}. \end{equation}
Let $\alpha=-\gamma\neq 0,\beta=0$  in \eqref{eq4.5} and using \eqref{rel-ab-A}, one has
\begin{equation}\label{rel-ab-4.10}
A_{0,j,i+k}=0,\quad \forall\ i,j,k \in
 \mathbb{Z}.
\end{equation}
Combining \eqref{rel-ab-A} and \eqref{rel-ab-4.10}, we have
\begin{equation}\label{rel-ab-4.11}
A_{\alpha,i,j+k}=A_{\alpha,j,i+k}, \quad \forall\ \alpha\in\Gamma,i,j,k \in\Z.\end{equation}
This shows for any given $\al\in\Gamma$, the value $A_{\al,i,j}$ only depends on  the sum $i+j$. So we shortly write $A_{\al,i,j}$ as $A_{\al,i+j}$ in what follows.

\begin{lemm} \label{lem4.2}
$A_{\alpha,\beta,i,j}=\frac{\alpha^2-2\alpha s}{8s^2}\delta_{\alpha+\beta,0}A_{4s,i+j}\
for\ all\ \alpha,\beta\in \Gamma$ and $i,j\in\Z$.
\end{lemm}
\begin{proof} It follows from  \eqref{definedn13} and by setting $\al=i=0$ in \eqref{eq4.5} that
$(\beta+\gamma)\varphi(L_{\beta,j},M_{\gamma,k})=0$, which gives
$A_{\beta,\gamma,j,k}=\varphi(L_{\beta,j},M_{\gamma,k})=0$ whenever $\beta+\gamma\neq 0$. So next we only need to consider the case $\al+\beta+\gamma=0$ in \eqref{eq4.5}.

Putting
$\beta=2ps,\gamma=-(\al+2ps)$ in \eqref{eq4.5} and using
\eqref{rel-ab-4.11} gives \begin{equation}\label{eq--4..9}
(\al+2ps)A_{\al,i}=(\al-2ps)A_{\al+2ps,i}+(\al+2ps)A_{2ps,i}, \quad \forall\ \al\in\Gamma, i,p\in\Z.
\end{equation}Note that $A_{2s,i}=A_{2s,0,i}=0$ for all $i\in\Z$ by \eqref{definedn13}. So in the case $p=1$ and $\al=2qs\ (q\in\Z),$ \eqref{eq--4..9} simply becomes $(q+1)A_{2qs,i}=(q-1)A_{2(q+1)s,i}$, in particular from which we derive  $A_{0,i}=-A_{2s,i}=0$ for all $i\in\Z.$ Thus, setting $p=1$ in \eqref{eq--4..9} gives \begin{equation}\label{eq--4..10}
(\al+2s)A_{\al,i}=(\al-2s)A_{\al+2s,i}, \quad \forall\ \al\in\Gamma, i\in\Z.
\end{equation}
Replacing $\al$ by $\al+2s$ in above equation we get  \begin{equation}\label{eq--4..11}
(\al+4s)A_{\al+2s,i}=\al A_{\al+4s,i}, \quad \forall\ \al\in\Gamma, i\in\Z.
\end{equation} While setting $p=2$ in \eqref{eq--4..9} we get another relation \begin{equation}\label{eq--4..12}
(\al+4s)A_{\al,i}=(\al-4s)A_{\al+4s,i}+(\al+4s)A_{4s,i}, \quad \forall\ \al\in\Gamma, i\in\Z.
\end{equation}Now it follows from  \eqref{eq--4..10}$-$\eqref{eq--4..12} and taking \eqref{rel-ab-4.11} into account that
$
A_{\alpha,i}=\frac{\alpha^2-2\alpha s}{8s^2}A_{4s,i}$ for $
\al\in\Gamma,$ $i\in\Z.
$ 
This completes the proof.   \end{proof}

\begin{lemm} \label{lem4.3}
$A_{\alpha,\beta,i,j}=B_{\alpha+s,\beta+s,i,j}=0, \quad\forall\
\alpha,\beta\in \Gamma,i,j\in\Z$.
\end{lemm}
\begin{proof}It follows from  Lemma \ref{lem4.2} and by setting  $\alpha=k=0$ in \eqref{eq4.6} that $$
(\beta+\gamma+2s)B_{\beta+s,\gamma+s,i,j}=0.$$ Thus $
B_{\beta+s,\gamma+s,i,j}=0$ 
{\rm if} $\beta+\gamma+2s\neq0.$ 
While
setting $\gamma=-2s$ and $\alpha=-\beta$ in \eqref{eq4.6}, we obtain
\begin{eqnarray}\label{eq4.14}
(2s+\beta)A_{-\beta,i+j+k}=-(s+\frac{3\beta}{2})B_{s,i+j,k}+(s-\frac{\beta}{2})B_{\beta+s,j,i+k},\quad \forall\ \beta\in\Gamma, i,j,k\in\Z.
\end{eqnarray}
In particular, setting $\beta=2s$ in above equation gives \begin{equation}\label{eq-B_{s,j,j}} B_{s,i,j}=-A_{-2s,i+j}=-A_{4s,i+j},\end{equation} where the latter equality follows from Lemma \ref{lem4.2}.
Now it follows from Lemma \ref{lem4.2}, \eqref{eq4.14} and \eqref{eq-B_{s,j,j}} that
\begin{eqnarray}\label{eq4.18}
B_{\beta+s,j,i+k}=-\frac{4s^2+6\beta s+\beta^2}{4s^2}A_{4s,i+j+k}
\quad {\rm for}\   \beta\neq 2s.\end{eqnarray}
For the case $\beta=2s$, setting $\alpha=4s, \beta=-2s$ and $\gamma=-4s$
in \eqref{eq4.6} and by \eqref{eq4.18}, one has \begin{equation}\label{B_{3s}}B_{3s,\,i+j,\,k}=-A_{4s,i+j+k},\quad \forall\ i,j,k\in\Z.\end{equation}
In particular,  setting  $\beta=4s$ and $\beta=-6s$ in \eqref{eq4.18},  we  respectively get
\begin{eqnarray}\label{eqnaaaaa4.20}
B_{5s,j,i+k}=-11A_{4s,i+j+k} \quad{\rm and}\quad B_{-5s,j,i+k}=-A_{4s,i+j+k},\quad \forall\ i,j,k\in\Z.
\end{eqnarray}
Note that $\varphi(Y_{5s,i},Y_{-5s,j})=-\varphi(Y_{-5s,j},Y_{5s,i})$. Hence, $B_{5s,\,i,\,j}=-B_{-5s,\,j,\,i}$
for any $i,j\in\Z$, which together with \eqref{eqnaaaaa4.20}  gives rise to $A_{4s,l}=0$ for
 all $l\in\Z$. This in turn forces for any $\al, \beta\in\Gamma$ and $i,j\in\Z$, $A_{\al,\beta,i,j}=0$
 by Lemma \ref{lem4.2} and $B_{\beta+s,i,j}=0$  by \eqref{eq4.18} and \eqref{B_{3s}}, completing the proof.
\end{proof}

\begin{lemm}\label{lem4.4}
$\varphi(M_{\alpha,i},M_{\beta,j})=\varphi(L_{\alpha,i},Y_{\beta+s,j})=\varphi(M_{\alpha,i},Y_{\beta+s,j})=0,\quad \forall\ \alpha,\beta\in
\Gamma,i,j\in\Z$.
\end{lemm}
\begin{proof}
It follows from  applying the Jacobi identity to the triple
$(L_{0,0},M_{\alpha,i},M_{\beta,j})$ and using  \eqref{rel5} that
$\varphi([L_{0,0},M_{\alpha,i}],M_{\beta,j})+\varphi(M_{\alpha,i},[L_{0,0},M_{\beta,j}])=0,$
 which gives
 \begin{eqnarray}\label{eq4.13}
\varphi(M_{\alpha,i},M_{\beta,j})=0 &\mbox{if \ } \alpha+\beta\neq0.
\end{eqnarray}
 On the
other hand, from the Jacobi identity on the triple
$(L_{-\alpha-2s,0},M_{\alpha,i},M_{2s,j})$ we get
$\varphi\big([L_{-\alpha-2s,0},M_{\alpha,i}],M_{2s,j}\big)+\varphi\big(M_{\alpha,i},[L_{-\alpha-2s,0},M_{2s,j}]\big)=0,$ i.e.,
\begin{eqnarray*}\label{eq2s4.21}
2s\varphi(M_{\alpha,i},M_{-\alpha,j})=-\alpha\varphi(M_{-2s,i},M_{2s,j}), \quad\forall\ \al\in\Gamma, i,j\in\Z.
\end{eqnarray*}
While it follows from  the Jacobi identity on the triple $(Y_{-3s,0}, Y_{s,i},
M_{2s,j})$ that \begin{eqnarray*}
4s\varphi(M_{-2s,i},M_{2s,j})=0, \quad\forall\ i,j\in\Z.
\end{eqnarray*} Hence, $
\varphi(M_{\alpha,i},M_{-\alpha,j})=0$ for $
\al\in\Gamma, i,j\in\Z.
$ 
Combining  this with \eqref{eq4.13} 
gives $$\varphi(M_{\alpha,i},M_{\beta,j})=0,\quad\forall\ \alpha,\beta\in \Gamma,i,j\in\Z,$$ as desired.

Applying the Jacobi identity on the triple
$(L_{0,0},L_{\alpha,i},Y_{\beta+s,j}),$ we get
\begin{eqnarray*}
(\beta+s-\frac\al2)\varphi(L_{0,0}, Y_{\al+\beta+s,i+j})&=&\varphi\big(L_{0,0},[L_{\alpha,i},Y_{\beta+s,j}]\big)\\
&=&\varphi\big(L_{\alpha,i},[L_{0,0},Y_{\beta+s,j}]\big)+\varphi\big([L_{0,0},L_{\alpha,i}],Y_{\beta+s,j}\big)\\
&=&(\beta+s)\varphi(L_{\al,i}, Y_{\beta+s,j})+\al\varphi(L_{\al,i}, Y_{\beta+s,j})\\
&=&(\al+\beta+s)\varphi(L_{\al,i}, Y_{\beta+s,j}).
\end{eqnarray*}
This together with
$\varphi(L_{0,0},Y_{\alpha+s,i})=0 $ for all $\alpha\in \Gamma$ and $i\in\Z $ (cf. \eqref{definedn13}) forces
  $$\varphi(L_{\alpha,i},Y_{\beta+s,j})=0,\quad \forall\ \al, \beta\in\Gamma, i,j\in\Z,$$  since $0\neq \alpha+\beta+s$ for all $\al,\beta\in\Gamma$. Similarly,  from the Jacobi identity on the triple
$(L_{0,0},M_{\alpha,i},Y_{\beta+s,j})$, we have
   $(\alpha+\beta+s)\varphi(M_{\alpha,i},Y_{\beta+s,j})=0$ and therefore
 $\varphi(M_{\alpha,i},Y_{\beta+s,j})=0$ for all $\al,\beta\in\Gamma$ and $i,j\in\Z$.
The proof of this lemma is complete.\end{proof}

It follows from   \cite[Theorem 4.3]{WWY} and  Lemmas \ref{lem4.2}--\ref{lem4.4} that we get another main result of the present paper.
\begin{theo} \label{th4.6}
The second cohomology group
\begin{eqnarray*}
H^{2}(\W,\ \mathbb{F})=\mbox{$\prod\limits_{k \in \mathbb{Z}}$}
\mathbb{F}\bar\phi_{k}
\end{eqnarray*}
is the direct product of all $\mathbb{F}\bar\phi_{k},\ k\in
\mathbb{Z}$, where $\bar\phi_{k}$ is the cohomology class of
$\phi_{k}$ defined by
\begin{eqnarray*}
\phi_{k}(L_{\alpha,i},L_{\beta,j})=\delta_{\alpha+\beta,0}\delta_{i+j,k}\frac{\alpha^{3}-\alpha}{12},
\quad \forall\ \alpha, \beta\in \Gamma,i,j\in \mathbb{Z}
\end{eqnarray*}
$($all other components vanish$)$.
\end{theo}

Consider the central extension $\widetilde{\W}=\W\oplus \big(C\otimes \mathbb F[t,t^{-1}]\big)$, whose Lie brackets are given as in Section 1 except the first relation in \eqref{rel1} is replaced by \begin{eqnarray*} [L_{\alpha,i},L_{\beta,j}]= (\beta-\alpha)L_{\alpha+\beta,i+j}+\delta_{\alpha+\beta,0}\frac{\alpha^{3}-\alpha}{12}C_{i+j},\nonumber\end{eqnarray*} and in addition, \begin{equation}\label{Ext-}  [\widetilde{\W},C_k]=0 \mbox{ \  for all $\alpha,\beta\in\Gamma$ and $i,j,k\in\Z$,} \end{equation}where $C_k=C\otimes t^{k}$. Theorem \ref{th4.6} implies that $\widetilde{\W}$ is the universal central extension of $\W$.

\end{CJK*}
\end{document}